\documentclass[a4paper,11pt]{amsart}
\usepackage{amsmath,amsfonts,amsthm,amssymb,enumerate,ifpdf}
\usepackage[pdftex]{graphicx}

\numberwithin{equation}{section}
\numberwithin{figure}{section}
\newtheorem{theorem}{Theorem}[section]
\newtheorem{corollary}[theorem]{Corollary}
\newtheorem{proposition}[theorem]{Proposition}

\newtheorem{lemma}[theorem]{Lemma}
\newtheorem{remark}[theorem]{Remark}

{
\theoremstyle{definition}
\newtheorem{definition}[theorem]{Definition}

}

\newenvironment{rem}{\begin{remark}\rm}{\end{remark}}

\title{Satellites of an oriented surface link and their local moves}
\author{Inasa Nakamura}
\address{
Institute for Biology and Mathematics of Dynamical Cell Processes (iBMath), Interdisciplinary Center for Mathematical Sciences, Graduate School of Mathematical Sciences, The University of Tokyo\newline
3-8-1 Komaba, Tokyo 153-8914, Japan 
}
 
\email{inasa@ms.u-tokyo.ac.jp}



\begin{document}

\begin{abstract}
For an oriented surface link $F$ in $\mathbb{R}^4$, 
we consider a satellite construction of a surface link, called a 2-dimensional braid over $F$, which is in the form of a covering over $F$. We introduce the notion of an $m$-chart on a surface diagram $\pi(F)\subset \mathbb{R}^3$ of $F$, which is a finite graph on $\pi(F)$ satisfying certain conditions and is an extended notion of an $m$-chart on a 2-disk presenting a surface braid. 
A 2-dimensional braid over $F$ is presented by an $m$-chart on $\pi(F)$. 
It is known that two surface links are equivalent if and only if their surface diagrams are related by a finite sequence of ambient isotopies of $\mathbb{R}^3$ and local moves called Roseman moves. 
We show that Roseman moves for surface diagrams with $m$-charts can be well-defined.
\end{abstract}
 
 \maketitle

\section{Introduction}%%%%%%%%
A {\it surface link} is the image of an embedding of a closed surface into the Euclidean 4-space $\mathbb{R}^4$. 
Two surface links are {\it equivalent} if one is carried to the other by an ambient isotopy of $\mathbb{R}^4$. 
A {\it surface diagram} of a surface link is its projected generic image in $\mathbb{R}^3$ equipped with over/under information along each double point curve. It is known \cite{Roseman} that two surface links are equivalent if and only if their surface diagrams are related by a finite sequence of local moves called Roseman moves as illustrated in Fig. \ref{fig:0417-03}, and ambient isotopies of the diagrams in $\mathbb{R}^3$.

\begin{figure}
\centering
\includegraphics*{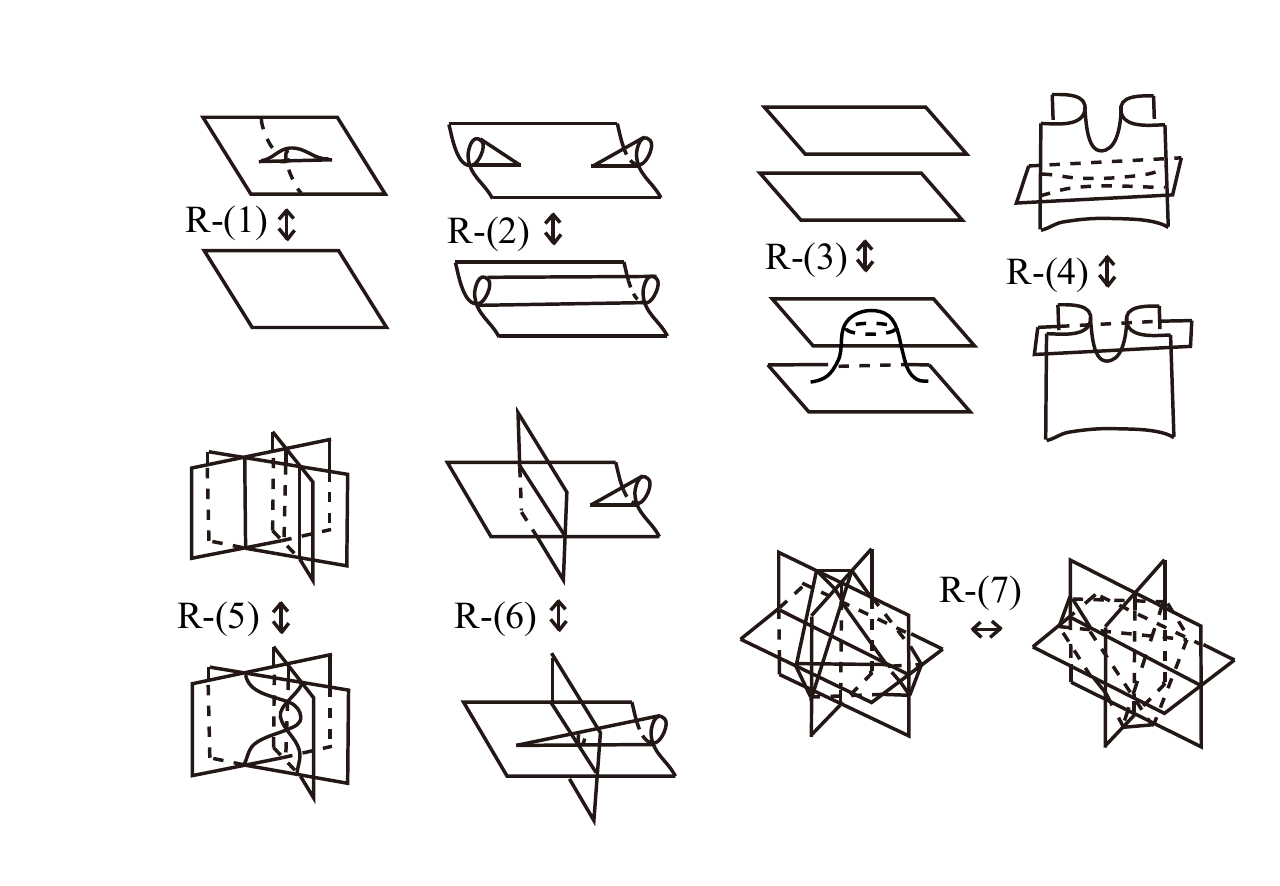}
\caption{Roseman moves. For simplicity, we omit the over/under information of each sheet. }
\label{fig:0417-03}
 \end{figure}%%

In this paper, we assume that surface links are oriented.  
We consider a satellite construction of a surface link $F$, which is in the form of a covering over $F$. In order to describe this construction, we introduce the notion of a 2-dimensional braid over $F$, which is an extended notion of a 2-dimensional braid over a 2-disk. 
A 2-dimensional braid over a 2-disk or a closed surface has a graphical presentation called an $m$-chart. 
We extend this notion to the notion of an $m$-chart on a surface diagram, and we show that a 2-dimensional braid over $F$ is presented by an $m$-chart on its surface diagram (Theorem \ref{thm:5-6}). 
Our main aim is to show that Roseman moves for surface diagrams with $m$-charts can be well-defined, by adding several new local moves to the original Roseman moves (Theorem \ref{prop:7-1}). 

The paper is organized as follows. 
In Section \ref{sec2}, we give the definition of 2-dimensional braids over a surface link. 
In Section \ref{sec3}, we review surface diagrams. 
In Section \ref{sec4}, we review chart description of 2-dimensional braids over a closed surface. 
In Section \ref{sec5}, we introduce the notion of an $m$-chart on a surface diagram and its presenting 2-dimensional braid, and we show Theorem \ref{thm:5-6}.
In Section \ref{sec6}, we introduce several new moves of Roseman moves for surface diagrams with $m$-charts. 
In Section \ref{sec7}, we give several remarks. 
Section \ref{sec8} is devoted to showing Theorem \ref{prop:7-1}.

\section{Two-dimensional braids over a surface link}\label{sec2}%%%%%%%
A 2-dimensional braid, which is also called a simple braided surface, over a 2-disk, is an analogous notion of a classical braid \cite{Kamada92,Kamada02,Rudolph}. 
We can modify this notion to a 2-dimensional braid over a closed surface \cite{N}. 
For a surface link $F$, we consider a closed surface embedded into $\mathbb{R}^4$ preserving the braiding structure as a satellite with companion $F$; see \cite[Section 2.4.2]{CKS}. We will call this a 2-dimensional braid over $F$.

Let $\Sigma$ be a closed surface, let $B^2$ be a 2-disk, and let $m$ be a positive integer.
\begin{definition}
A closed surface $S$ embedded in $B^2 \times \Sigma$ is called a {\it 2-dimensional braid over $\Sigma$} of degree $m$ if it satisfies the following conditions:
\begin{enumerate}[(1)]
\item
The restriction $p|_{S} \,:\, S \rightarrow \Sigma$ is a branched covering map of degree $m$, where $p\,:\, B^2 \times \Sigma \to \Sigma$ is the projection to the second factor. 
\item
For each $x \in \Sigma$, $\#(S \cap p^{-1}(x))=m-1$ or $m$.
\end{enumerate} 

Take a base point $x_0$ of $\Sigma$. 
Two 2-dimensional braids over $\Sigma$ of degree $m$ are {\it equivalent} if there is a fiber-preserving ambient isotopy of $B^2 \times \Sigma$ rel $p^{-1}(x_0)$ which carries one to the other. 

A {\it trivial} 2-dimensional braid is the product of $m$ distinct interior points of $B^2$ and $\Sigma$ in $B^2 \times \Sigma$.
\end{definition}

For a closed surface $\Sigma$, a surface link is said to be {\it of type $\Sigma$} when it is the image of an embedding of $\Sigma$. 
Let $F$ be a surface link of type $\Sigma$, and let $N(F)$ be a tubular neighborhood of $F$ in $\mathbb{R}^4$. 

\begin{definition}\label{def:2-braid}
Let $S$ be a 2-dimensional braid over $\Sigma$ in $B^2 \times \Sigma$. 
Let $f: B^2 \times \Sigma \to \mathbb{R}^4$ be an embedding so that $f(B^2 \times \Sigma)=N(F)$. 
Then we call the image $f(S)$ a {\it 2-dimensional braid over $F$}. We call $F$, $S$, and $f$ the {\it companion}, {\it pattern}, and {\it associated embedding} of the 2-dimensional braid, respectively.
We define the {\it degree} of $f(S)$ as that of $S$.

Let $\tilde{S}$ (resp. $\tilde{S}^\prime$) be a 2-dimensional braid of degree $m$ over a surface link $F$ (resp. $F^\prime$) of type $\Sigma$ with pattern $S$ (resp. $S^\prime$) and the  associated embedding $f$ (resp. $f^\prime$). 
Two 2-dimensional braids $\tilde{S}$ and $\tilde{S}^\prime$ are {\it equivalent} if there is an ambient isotopy $\{g_u\}_{u \in [0,1]}$ of $\mathbb{R}^4$ which carries one to the other such that $g_1 \circ f=f^\prime$ and $S$ and $S^\prime$ are equivalent.

When $S$ is trivial, we call $f(S)$ a {\it trivial} 2-dimensional braid. 
\end{definition}
Note that if, for surface links $F$ and $F^\prime$, two 2-dimensional braids over them are equivalent, then $F$ and $F^\prime$ are equivalent surface links. Further, equivalent 2-dimensional braids are also equivalent surface links. Hence, the following proposition follows.
\begin{proposition}
Let $f(S)$ and $f(S^\prime)$ be 2-dimensional braids over a surface link $F$ with the same associated embedding $f:B^2 \times \Sigma \to N(F)$ and with pattern $S$ and $S^\prime$, respectively. 
If $S$ and $S^\prime$ are equivalent, 
then $f(S)$ and $f(S^\prime)$ are equivalent both as 2-dimensional braids and as surface links.
\end{proposition}

From now on throughout the paper, when we say that \lq\lq 2-dimensional braids are equivalent", we mean in the sense of Definition \ref{def:2-braid} except when explicitly indicated otherwise.  

A 2-dimensional braid $f(S)$ over $F$ is a specific satellite with pattern $S$ and companion $F$; see \cite[Section 2.4.2]{CKS}, see also \cite[Chapter 1]{Lickorish}. Satellite constructions have been used for constructing knotted surfaces 
with desired properties for the fundamental groups $G$ of the complement; see Litherland \cite{Lith81} for knotted surfaces with non-trivial $H_2(G)$, and Kanenobu \cite{Kane83} for knotted 2-spheres such that each group $G$ is prime and has deficiency $-n$ for a given positive integer $n$. A 2-dimensional braid over a surface link is an extended notion of cabling for knotted surfaces; for the study on cabling, see Kanenobu \cite{Kane85} (resp. Iwase \cite{Iwa88, Iwa90}) for the case when the companion is a knotted sphere (resp. a standard torus), and Hirose \cite{Hiro03} for higher genus cases. In \cite{N, N2}, we studied surface links  which are in the form of 2-dimensional braids over a standard torus; see also our remark in Section \ref{sec7-2}. Kamada \cite{Kama90a, Kama90b} studied doubled surfaces, which are satellites whose companion is a nonorientable surface.

\section{Review of surface diagrams}\label{sec3}
We will review a surface diagram of a surface link $F$; see \cite{CKS}. For a projection $\pi \,:\, \mathbb{R}^4 \to \mathbb{R}^3$, the closure of the self-intersection set of $\pi(F)$ is called the singularity set. Let $\pi$ be a generic projection, i.e. the singularity set of the image $\pi(F)$ consists of double points, isolated triple points, and isolated branch points; see Fig. \ref{0215-1}. The closure of the singularity set forms a union of immersed arcs and loops, which we call double point curves. Triple points (resp. branch points) form the intersection points (resp. the end points) of the double point curves. A {\it surface diagram} of $F$ is the image $\pi(F)$ equipped with over/under information along each double point curve with respect to the projection direction. 
We denote the surface diagram by the same notation $\pi(F)$.
A surface diagram $\pi^\prime(S)$ of a 2-dimensional braid $S$ over $\Sigma$ in $I \times I \times \Sigma$ for an interval $I$ is defined likewise, where $\pi^\prime: I \times I \times \Sigma \to I \times \Sigma$ is a generic projection.

\begin{figure}
\begin{center}
 \includegraphics*{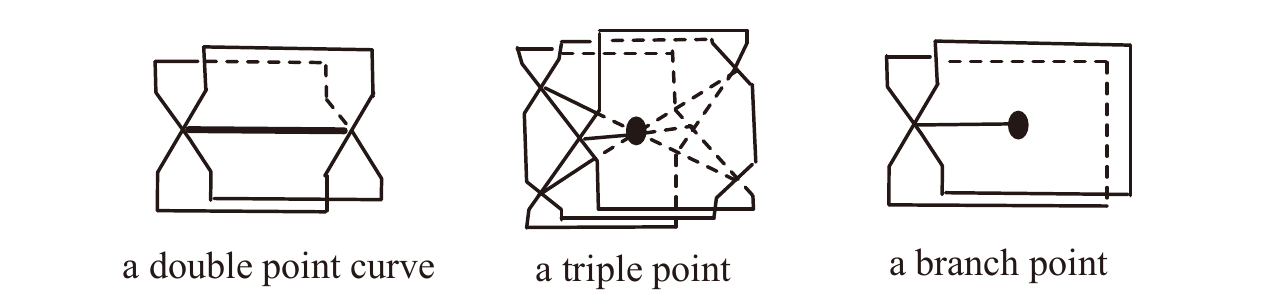}
\end{center}
  \caption{The singularity of a surface diagram.}
  \label{0215-1}
  \end{figure}

Let $S$ be a 2-dimensional braid over a surface link $F$. 
Throughout this paper, we consider the surface diagram of $S$ by a projection $\pi:\mathbb{R}^4 \to \mathbb{R}^3$ which is generic with respect to the companion $F$. Perturbing $S$ if necessary, we can assume that $\pi$ is also generic with respect to $S$. 
In particular, if the surface diagram $\pi(F)$ has no singularities, then by an ambient isotopy of $\mathbb{R}^4$ satisfying the conditions of equivalence of $S$, we can assume that $F$ is standard, i.e. $F$ is the boundary of a handlebody in $\mathbb{R}^3 \times \{0\}$. Then we assume that $\pi$ projects $N(F)=I \times I \times F$ to $I \times F$, where we identify $N(F)$ with $I \times I \times F$ in such a way as follows. Since $F$ is the boundary of a handlebody in $\mathbb{R}^3 \times \{0\}$, the normal bundle of $F$ in $\mathbb{R}^3 \times \{0\}$ is a trivial bundle. We identify it with $I \times F$. Then we identify $N(F)$ with $I \times I \times F$, where the second $I$ is an interval in the fourth axis of $\mathbb{R}^4$. In this case, we can identify $\pi^\prime(S)$ with $\pi(S)$. 

\section{Review of chart description of 2-dimensional braids over $\Sigma$} \label{sec4}%%%%% 
 The graphical method called an $m$-chart on a 2-disk was introduced to present a simple surface braid which is a 2-dimensional braid over a 2-disk with trivial boundary condition \cite{Kamada92, Kamada02}. By regarding an $m$-chart on a 2-disk as drawn on a 2-sphere $S^2$, it presents a 2-dimensional braid over $S^2$ \cite{Kamada92, Kamada02, N}. 
These notions can be modified for $m$-charts on $\Sigma$ \cite{N}; see also \cite{Kamada02}. An $m$-chart on $\Sigma$ presents a 2-dimensional braid over $\Sigma$ \cite{N}. 
\\
 
We identify $B^2$ with $I \times I$. 
When a 2-dimensional braid $S$ over $\Sigma$ is given, we obtain a graph on $\Sigma$, as follows. 
We can assume that the projection $\pi^\prime: I \times I \times \Sigma \to I \times \Sigma$ is generic with respect to $S$. 
Then the singularity set $\mathrm{Sing}(\pi^\prime(S))$ of $\pi^\prime(S)$ consists of double point curves, triple points, and branch points. Moreover we can assume that the singular set of the image of $\mathrm{Sing}(\pi^\prime(S))$ by the projection to $\Sigma$ consists of a finite number of double points such that the preimages belong to double point curves of $\mathrm{Sing}(\pi^\prime(S))$. 
Thus the image of $\mathrm{Sing}(\pi^\prime(S))$ by the projection to $\Sigma$ forms a finite graph $\Gamma$ on $\Sigma$ such that the degree of its vertex is either $1$, $4$ or $6$. 
An edge of $\Gamma$ corresponds to a double point curve, and a vertex of degree $1$ (resp. $6$) corresponds to a branch point (resp. a triple point). 

For such a graph $\Gamma$ obtained from a 2-dimensional braid $S$, we give orientations and labels to the edges of $\Gamma$, as follows. 
Let us consider a path $l$ in $\Sigma$ such that $\Gamma \cap l$ is a point $P$ of an edge $e$ of $\Gamma$. 
Then $S \cap \pi^{\prime -1} (l)$ is a classical $m$-braid with one crossing in $\pi^{\prime -1}(l)$ such that $P$ corresponds to the crossing of the $m$-braid. Let $\sigma_{i}^{\epsilon}$ ($i \in \{1,\ldots, m-1\}$, $\epsilon \in \{+1, -1\}$) be the presentation of $S \cap \pi^{\prime -1}(l)$. 
Then label the edge $e$ by $i$, and moreover give $e$ an orientation such that the normal vector of $l$ coincides (resp. does not coincide) with the orientation of $e$ if $\epsilon=+1$ (resp. $-1$). We call such an oriented and labeled graph an {\it $m$-chart of $S$}. 
\\
 
 In general, we define an $m$-chart on $\Sigma$ as follows. 

\begin{definition}\label{def:chart}%%%%%%%%%%%%%%%
A finite graph $\Gamma$ on $\Sigma$ is called an {\it m-chart on $\Sigma$} if it satisfies the following conditions: 

\begin{enumerate}[(1)]
 \item Every edge is oriented and labeled by an element of $\{1,2, \ldots, m-1\}$. 
 \item Every vertex has degree $1$, $4$, or $6$.
 \item  The adjacent edges around each vertex are oriented and labeled as shown in Fig. \ref{Fig1-1}, where we depict a vertex of degree 1 (resp. 6) by a black vertex (resp. a white vertex).
 \end{enumerate}

An {\it empty $m$-chart} is an $m$-chart presented by an empty graph.
 \end{definition}%%%%%%%%%%%%%%
 
 \begin{figure}%%%%%%%%%%%%
 \includegraphics*{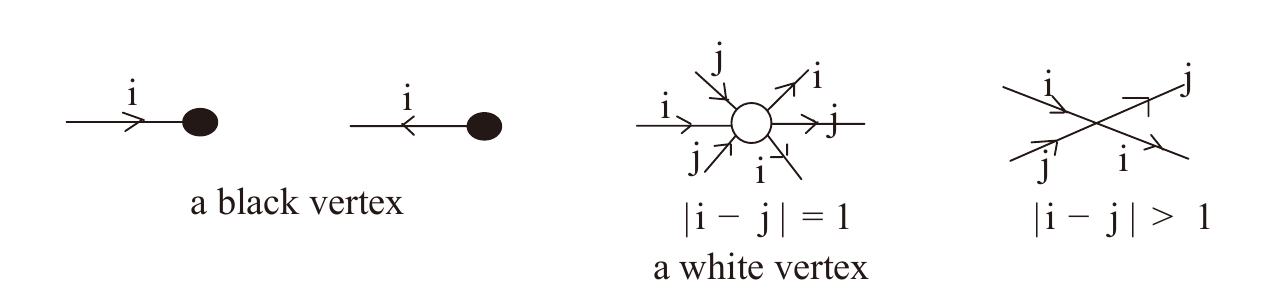}
 \caption{Vertices in an $m$-chart.}
 \label{Fig1-1}
 \end{figure}%%%%%%%%%%%%%%%%
  
 When an $m$-chart $\Gamma$ on $\Sigma$ is given, we can reconstruct a 2-dimensional braid $S$ over $\Sigma$ such that the original $m$-chart $\Gamma$ is an $m$-chart of $\Sigma$; see \cite{Kamada92-2, Kamada02, N}. 
A black vertex (resp. a white vertex) of $\Gamma$ presents a branch point (resp. a triple point) of $S$. 
\\

For a surface link $F$ whose surface diagram $D$ has no singularities, a 2-dimensional braid over $F$ is presented by an $m$-chart on $D$ (see Section \ref{sec3}). 

In order to define an $m$-chart on a surface diagram with singularities, we will define an $m$-chart $\Gamma$ over a compact surface $E$ with boundary, as follows: $\Gamma$ is a graph which can be regarded as $\tilde{\Gamma} \cap E$ for an $m$-chart $\tilde{\Gamma}$ on a closed surface $\tilde{E}$ containing $E$ such that $\tilde{\Gamma} \cap \partial E$ are transverse intersections with edges of $\tilde{\Gamma}$. A 2-dimensional braid over $E$ presented by $\Gamma$ is the surface $S \cap (B^2 \times E)$, where $S$ is a 2-dimensional braid in $B^2\times \tilde{E}$ presented by $\tilde{\Gamma}$ (see \cite{Kamada07}).

\section{Charts on surface diagrams and 2-dimensional braids} \label{sec5}
In this section, let $F$ be a surface link and let $D$ be a surface diagram of $F$. 
We define an $m$-chart on $D$ and its presenting 2-dimensional braid over $F$. First, in Section \ref{sec:5-1}, we define a 2-dimensional braid over $F$ presented by an empty $m$-chart on $D$. Then in Section \ref{sec:5-2}, we define an $m$-chart on $D$ and its presenting 2-dimensional braid. We will show that any 2-dimensional braid over $F$ is presented by an $m$-chart on $D$ (Theorem \ref{thm:5-6}). 

\subsection{An empty $m$-chart on a surface diagram}\label{sec:5-1}

A 2-dimensional braid is presented by a motion picture consisting of isotopic transformations and hyperbolic transformations. 
For a knotted surface (i.e. a properly embedded surface) $S$ in a 4-ball, for simplicity of the argument, we assume that $S$ is in $\mathbb{R}^4$.
A {\it motion picture} of a knotted surface $S \subset \mathbb{R}^4$ is a one-parameter family $\{ \pi(S \cap (\mathbb{R}^3 \times \{t\}) )\}_{t \in \mathbb{R}}$, where $\pi \,:\,\mathbb{R}^4 \to \mathbb{R}^3, (x,y,z,t) \mapsto (x,y,z)$, is the projection (see \cite{Kamada02}). 

Let $L$ be a classical link or knotted strings (i.e. a properly embedded strings) in a 3-ball. For simplicity of the argument, we assume that $L$ is in $\mathbb{R}^3$.
For an ambient isotopy $\{h_t\}_{t \in [0,1]}$ of $\mathbb{R}^3$ rel $\partial L$, we have an isotopy (a one-parameter family) $\{h_t(L)\}$ of classical links or knotted strings. We say that $h_1(L)$ is obtained from $L$ by an {\it isotopic transformation}, and we use the notation that $L \rightarrow h_1(L)$ is an isotopic transformation \cite[Section 9.1]{Kamada02}. 
When $\{h_t(L)\}$ is an isotopy of classical braids, we will say that $L \rightarrow h_1(L)$ is a {\it braid isotopic transformation}.
   
For a classical link or knotted strings $L$, a 2-disk $B$ in $\mathbb{R}^3$ is called a {\it band} attaching to $L$ if $L \cap B$ is a pair of disjoint arcs in $\partial B$. For mutually disjoint bands $B_1, \ldots, B_n$ attaching to $L$, putting $\mathcal{B}=B_1 \cup \cdots \cup B_n$, we define a new classical link or knotted strings $h(L;\mathcal{B})$ by the closure of $(L \cup \partial \mathcal{B})-(L \cap \mathcal{B})$,  where we take the closure as a subset of $\mathbb{R}^3$. We say that $h(L;\mathcal{B})$ is obtained from $L$ by a {\it hyperbolic transformation} along the bands $B_1, \ldots, B_n$, and we use the notation that $L \dot{\to} h(L;\mathcal{B})$ is a hyperbolic transformation \cite[Section 9.1]{Kamada02}. 

Let $\sigma_1, \sigma_2, \ldots$ be the standard generators of the classical braid group.  
Let us take the $2m$-braids as follows:
\begin{eqnarray*}
&& \Pi_i=\sigma_{m+1} \sigma_{m+2}\cdots \sigma_{m+i}, \ \Pi_i^\prime=\sigma_{m-1} \sigma_{m-2}\cdots \sigma_{m-i}, \\
&& \Delta=\Pi_{m-1}\Pi_{m-2} \cdots \Pi_{1}, \ \Delta^\prime=\Pi^\prime_{m-1}\Pi^\prime_{m-2} \cdots \Pi^\prime_{1}, \\
&& \Theta=\sigma_m \cdot \Pi^\prime_{m-1} \cdot \Pi_{m-1}\cdot \sigma_m\cdot  \Pi^\prime_{m-2} \cdot \Pi_{m-2}\cdots \sigma_m\cdot\Pi^\prime_{1} \cdot \Pi_{1}\cdot\sigma_m
\end{eqnarray*}
Note that $\Delta$ (resp. $\Delta^\prime$) is the $2m$-braid obtained from an $m$-braid with a half twist by adding $m$ trivial strings before (resp. after) it. 
    
\begin{definition}
For a surface diagram $D$, 
we can take a presentation of a local part of $D$ around (1) a regular point i.e. a nonsingular point, (2) a double point curve, and (3) a branch point, as follows, where (3) has two types, i.e. (3-1) a positive branch point and (3-2) a negative branch point (see \cite[Section 2.5.2]{CKS}): 
\begin{enumerate}
\item[(1)] A 2-disk $E$. 
\item[(2)] The product of a 2-braid $\sigma_1$ and an interval. 
\item [(3)]
A surface with a motion picture (3-1) $\sigma_1 \dot{\to} e_2$ 
or (3-2) $e_2 \dot{\to} \sigma_1$, where $\dot{\to}$ is a hyperbolic transformation along a band corresponding to the crossing of the 2-braid $\sigma_1$, and $e_2$ denotes the trivial $2$-braid. 
\end{enumerate}

We define the {\it 2-dimensional braid over $F$ presented by $D$ equipped with an empty $m$-chart} as a surface defined locally for such local parts of $D$, as follows.
\begin{enumerate}
\item[(1)]
Parallel $m$ copies of $E$. 
\item[(2)]
The product of the following $2m$-braid and the interval (see Fig. \ref{fig:0429-01}):
\begin{eqnarray*}
 (\Delta^{\prime})^{-1} \cdot \Delta^{-1} \cdot \Theta.
\end{eqnarray*}

\item[(3-1)]A surface presented by a motion picture as follows (see Fig. \ref{fig:0419-01}): 
\begin{eqnarray*} 
  (\Delta^{\prime})^{-1} \cdot \Delta^{-1} \cdot \Theta \,
 \dot{\to}\, 
  (\Delta^{\prime})^{-1} \cdot \Delta^{-1} \cdot \Delta^\prime \cdot \Delta  \to e_{2m},
\end{eqnarray*}
where $\to$ is a braid isotopic transformation and $\dot{\to}$ is a hyperbolic transformation along bands corresponding to the $m$ $\sigma_m$'s for the $2m$-braid, and $e_{2m}$ denotes the trivial $2m$-braid.

\item[(3-2)] A surface 
presented by a motion picture which is the inverse of (3-1).

\end{enumerate}
We construct a 2-dimensional braid over $F$, as follows. 
Since the surface diagram is oriented, we can extend the construction over the complement of small disks each of which is a neighborhood of a triple point. Since the boundary of each of these disks presents the closure of a classical $2m$-braid which is equivalent to the trivial $2m$-braid, we can extend the construction over the disks by pasting $2m$ disks over each disk. Thus we extend the construction over the whole surface link.  
\end{definition}

\begin{figure}
 \centering\includegraphics*{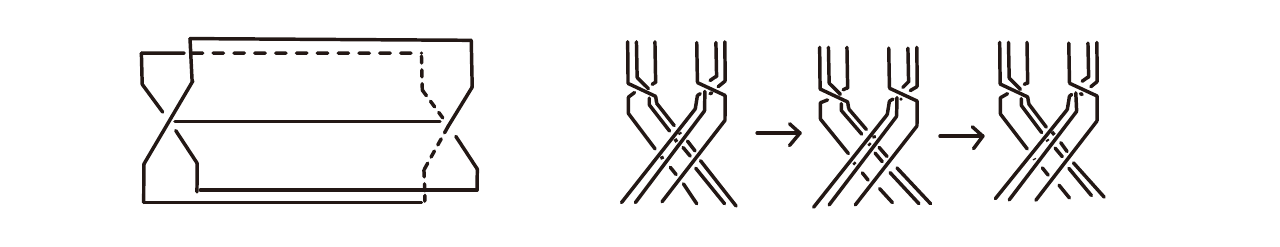}
\caption{An empty $m$-chart around a double point curve and its presenting 2-dimensional braid.}
\label{fig:0429-01}
 \end{figure}

\begin{figure}
 \includegraphics*{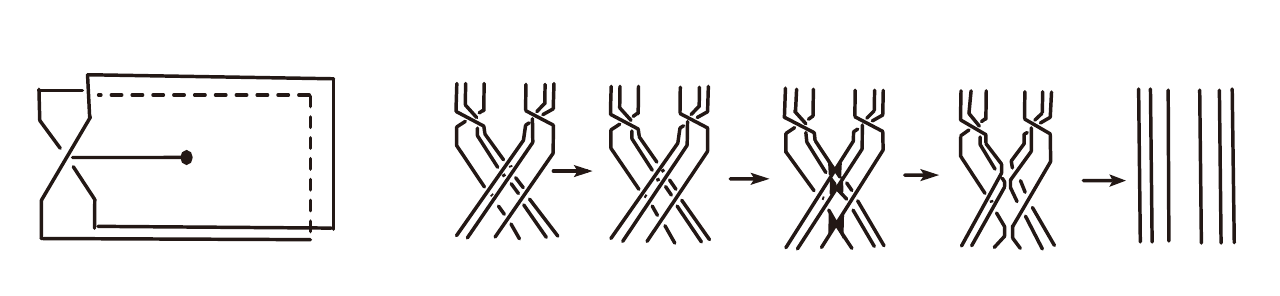}
\caption{An empty $m$-chart around a positive branch point and its presenting 2-dimensional braid.}
\label{fig:0419-01}
 \end{figure}%%

\begin{proposition}
The 2-dimensional braid presented by an empty $m$-chart on $D$ is well-defined. 
\end{proposition}
\begin{proof}
It suffices to show that the construction of the 2-dimensional braid is unique up to equivalence. The ambiguity comes from the braid isotopic transformations constructing the surfaces around double point curves and around triple points of $D$. These surfaces can be regarded as 2-dimensional braids over a 2-disk $E$. It is known (see \cite[Theorem 1.34]{CKS}, \cite{Rudolph}) that for 2-dimensional braids on $E$ of the same degree and with no branch points, they are equivalent if the boundaries are the same. Hence the construction is unique up to equivalence. 
\end{proof}

Let $\tilde{D}$ be the surface diagram of the 2-dimensional braid presented by an empty $m$-chart on $D$. 
Around a double point curve of $D$, $\tilde{D}$ has $m(m-1)+m^2=m(2m-1)$ double point curves. 

\subsection{An $m$-chart on a surface diagram and the 2-dimensional braid}\label{sec:5-2}
For a surface diagram $D$, let us denote its singularity set by $\mathrm{Sing}(D)$.

\begin{definition}\label{def:chart-D}
 
A finite graph $\Gamma$ on $D$ is called an {\it $m$-chart on $D$} if it satisfies the following conditions: 

\begin{enumerate}[(1)]
 \item For a compact surface $E \subset D-\mathrm{Sing}(D)$, $\Gamma \cap E$ is an $m$-chart on $E$ (see Section \ref{sec4}). 

\item
The intersection $\Gamma \cap \mathrm{Sing}(D)$ consists of a finite number of points, each of which forms a vertex of degree 2 of $\Gamma$ as indicated in Fig. \ref{fig:0417-01}, where $i \in \{1,\ldots, m-1\}$. 
 \end{enumerate}
\end{definition}
\begin{figure}[ht]
 \centering\includegraphics*{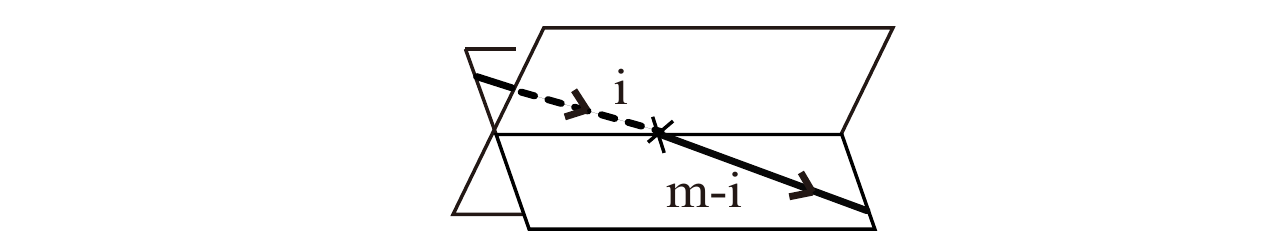}
\caption{An $m$-chart around a double point curve, where $i \in \{1,\ldots,m-1\}$. For simplicity, we omit the over/under information of each sheet.}
\label{fig:0417-01}
 \end{figure}

\begin{definition}\label{def:20131211}
Let $\Gamma$ be an $m$-chart on $D$. 
We define the {\it 2-dimensional braid presented by $\Gamma$}, as follows. 
\begin{enumerate}[(1)]

\item
For a compact surface $E \subset D-\mathrm{Sing}(D)$, the presented 2-dimensional braid over $E$ is the one presented by $\Gamma \cap E$ (see Section \ref{sec4}).  

\item
For a compact surface $E$ with $E \cap \mathrm{Sing}(D) \neq \emptyset$ and $\Gamma \cap E=\emptyset$, the 2-dimensional braid is the one presented by an empty $m$-chart on $E$. 
\item
Around a double point curve presented by the product of a 2-braid $\sigma_1$ and an interval such that the $m$-chart is as in the top figure in Fig. \ref{fig:0417-02}, the 2-dimensional braid is as follows:
\begin{eqnarray*}
 \sigma_i^{\epsilon} \cdot(\Delta^\prime)^{-1}\cdot (\Delta)^{-1} \cdot\Theta 
\to (\Delta^\prime)^{-1} \cdot(\Delta)^{-1} \cdot\sigma_{m-i}^{\epsilon}\cdot\Theta
\to (\Delta^\prime)^{-1} \cdot(\Delta)^{-1} \cdot\Theta \cdot\sigma_{2m-i}^{\epsilon},
\end{eqnarray*}
where $\epsilon=+1$ (resp. $-1$) if the edge of the $m$-chart is  oriented from upper left to bottom right (resp. bottom right to upper left), and $i \in \{1,\ldots,m-1\}$. 

The other cases are likewise; see Fig. \ref{fig:0417-02}. 
\end{enumerate}

\begin{figure}[ht]
 \centering\includegraphics*{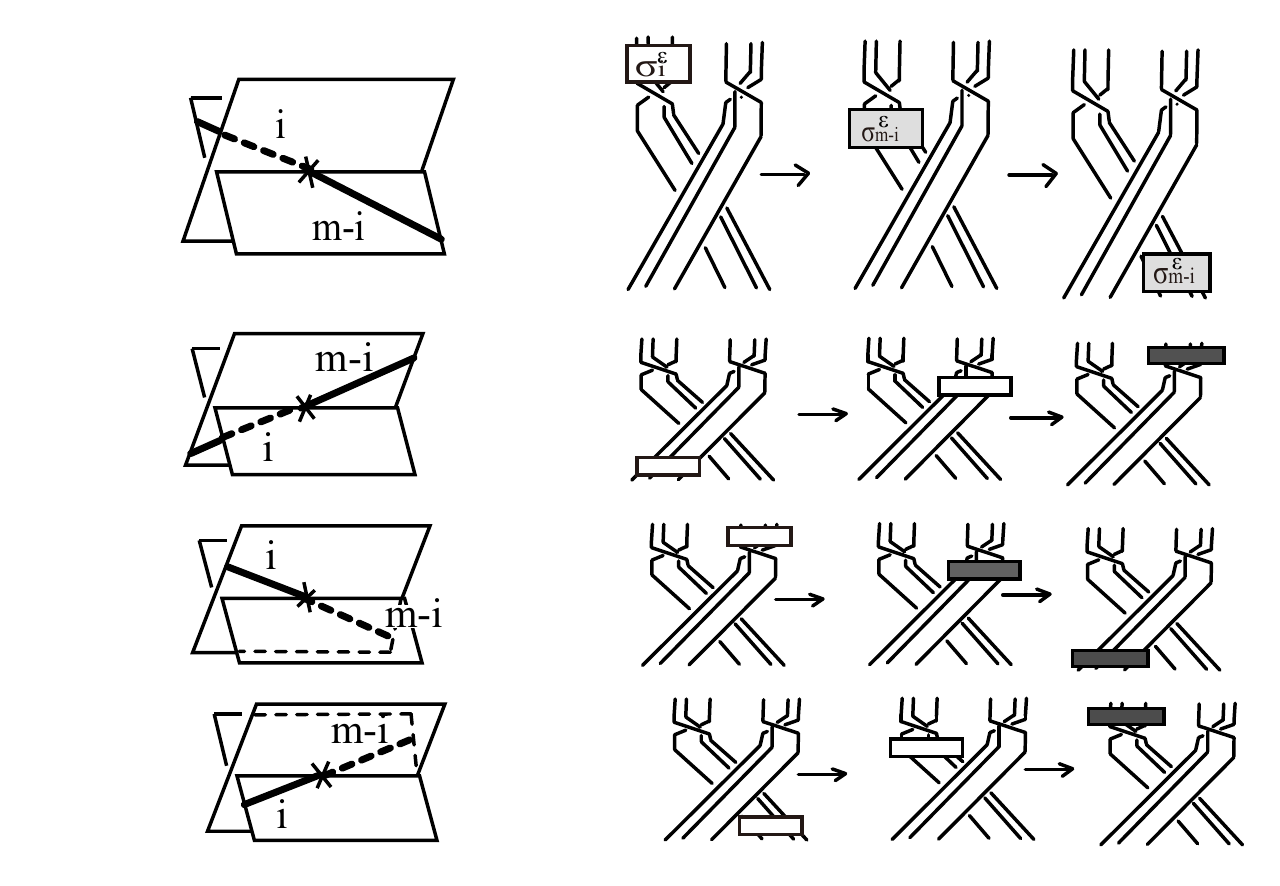}
\caption{An $m$-chart around a double point curve and its presenting 2-dimensional braid. We indicate by the white box (resp. black box) the $m$-braid $\sigma_i^\epsilon$ (resp. $\sigma_{m-i}^\epsilon$), where $i \in \{1,\ldots,m-1\}$ and $\epsilon \in \{+1, -1\}$. For simplicity, we omit the orientations of the edges.}
\label{fig:0417-02}
 \end{figure}%%
\end{definition}

\begin{theorem}\label{thm:5-6}
A 2-dimensional braid over a surface link $F$ is presented by an $m$-chart on its surface diagram $D$.
\end{theorem}

\begin{proof}
First we show that an empty $m$-chart on $D$ presents a trivial 2-dimensional braid over $F$.
Let us denote by $S$ the 2-dimensional braid presented by an empty $m$-chart on $D$, and let $Q_m$ be $m$ distinct interior points of $B^2$. We regard $S$ as a fiber bundle over $F$ in $N(F)=B^2 \times F$.  By definition, we can see that for each point $x$ of $F$, its fiber $S \cap (B^2 \times \{x\})$ is $Q_m$ as sets. 
Hence, in order to show that $S$ is trivial, it suffices to show that for any closed path $l$ embedded in $F$, the fibers $Q_m$ of the initial point and the terminal point of $l$ coincide pointwise. By perturbing $l$ if necessary, we can assume that the projection of $l$ on $D$ intersects with the singularity of $D$ at $n$ points as transverse intersections with double point curves of $D$. It is known that the complementary regions of $\mathbb{R}^3$ divided by $D$ admits a checkerboard coloring such that each region is colored either black or white and adjacent regions always have distinct colors (see \cite[Section 2.5.2]{CKS}). Hence, by considering a closed path $l^\prime$ obtained by shifting $l$ to a diagonal direction such that the intersections of $l^\prime$ and $D$ are $n$ transverse intersections, we can see that $n$ is even. It implies that $l$ presents a 1-dimensional closed braid $L$ in $l \times B^2 \subset N(F)$ with even numbers of half twists, i.e. with several full twists. Thus the initial points and the terminal points of $L$, i.e. the fibers of the initial point and the terminal point of $l$, coincide pointwise. Thus $S$ is a trivial 2-dimensional braid. 

Hence any 2-dimensional braid $S^\prime$ over $F$ is obtained from $S$ by adding the braiding information.  
Let $\Gamma$ be the projection to $D$ of the singularity set of the diagram of $S^\prime$. 
By perturbing $S^\prime$ if necessary, we can assume that the intersection of $\Gamma$ and the singularity set of $D$  consists of a finite number of transverse intersections of edges of $\Gamma$ and double point curves of $D$. 
Around each intersection, the 2-dimensional braid is as in Definition \ref{def:20131211} (3) (see also Fig. \ref{fig:0417-02}). Thus, by adding the chart information as in Definition \ref{def:chart-D} (2) to the edges of $\Gamma$ around the intersection, and adding the original chart information to the other edges, we obtain an $m$-chart on $D$ which presents $S^\prime$. Thus we have the required result. 
\end{proof}

\begin{rem}\label{rem:5-3}
Let $F$ be a surface link of type $\Sigma$.
Any 2-dimensional braid $S$ presented by an $m$-chart on $D$ has the same framing of $N(F)$ determined from an empty $m$-chart on $D$, where a {\it framing} (see \cite[Chapter 12]{Lickorish}) of $N(F)$ is a parametrization of $N(F)$ by $\{x\} \times \Sigma$ ($x \in B^2$), where $f(B^2 \times \Sigma) =N(F) \cong B^2 \times F$ for the associated embedding $f$ of $S$ (up to fiber-preserving ambient isotopy of $N(F) \cong B^2 \times F$ rel $B^2 \times \{x_0\}$ for a base point $x_0 \in F$).
\end{rem}

\section{Roseman moves for surface diagrams with $m$-charts}\label{sec6}
 
\begin{definition}
We define {\it Roseman moves for surface diagrams with $m$-charts} by the local moves as illustrated in Figs. \ref{fig:0417-03} and \ref{fig:0417-04}.  
\end{definition}
\begin{figure}
 \includegraphics*{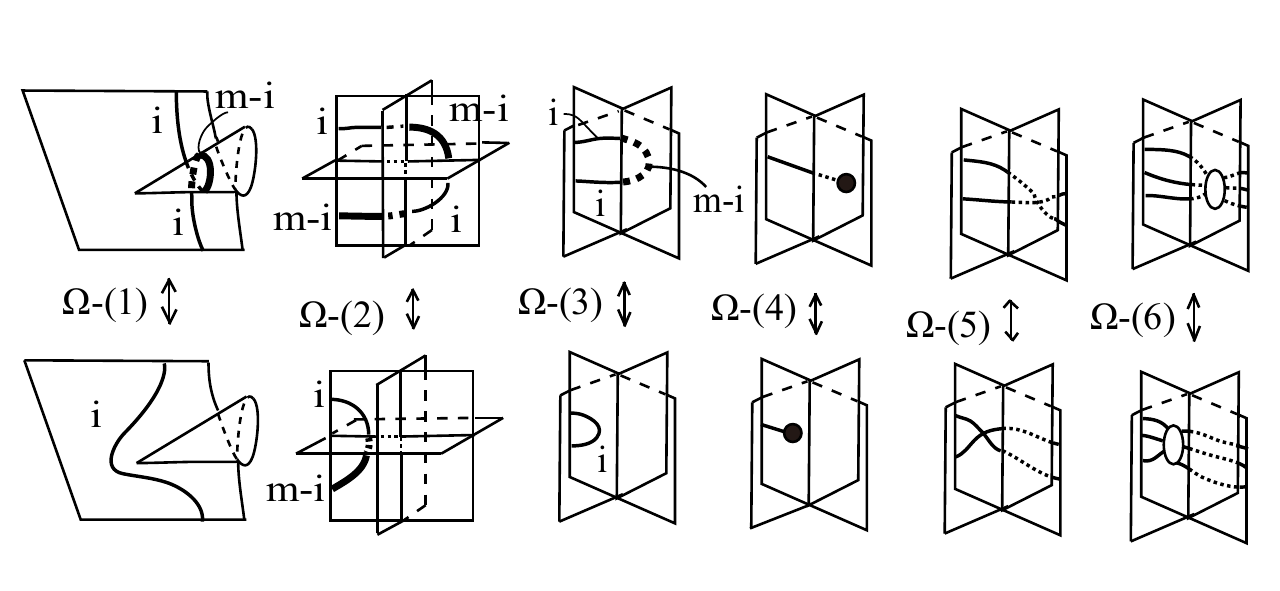}
\caption{Roseman moves for surface diagrams with $m$-charts, where $i \in \{1,\ldots,m-1\}$. For simplicity, we omit the over/under information of each sheet, and orientations and labels of edges of $m$-charts. }
\label{fig:0417-04}
 \end{figure}%
 
\begin{theorem}\label{prop:7-1}
Roseman moves for surface diagrams with $m$-charts as illustrated in Fig. \ref{fig:0417-03} and \ref{fig:0417-04} are well-defined, i.e. for each pair of Roseman moves, the $m$-charts on the indicated diagrams present equivalent 2-dimensional braids. 
\end{theorem}

We give the proof in Section \ref{sec8}. Theorem \ref{prop:7-1} implies the following corollary.
 
\begin{corollary}\label{thm:6-2}
Two 2-dimensional braids over a surface link $F$ of degree $m$  are equivalent if their surface diagrams with $m$-charts are related by a finite sequence of Roseman moves and ambient isotopies of $\mathbb{R}^3$.
\end{corollary}

For surface links $F$ and $F^\prime$,
if $F$ is equivalent to $F^\prime$, then there is an ambient isotopy of $\mathbb{R}^4$ carrying $N(F)=B^2 \times F$ to $N(F^\prime)=B^2 \times F^\prime$ as $B^2$-bundle over a surface. If $S$ is a 2-dimensional braid in $N(F)$, then $S$ is sent to a 2-dimensional braid in $N(F^\prime)$. Hence, together with Theorem \ref{prop:7-1}, we have the following remark. Let $D$ and $D^\prime$ be the surface diagrams of $F$ and $F^\prime$ respectively. 

\begin{rem}\label{thm:6-3}
For equivalent surface links $F$ and $F^\prime$, 
any 2-dimensional braid over $F$ can be deformed by equivalence to the form of a 2-dimensional braid over $F^\prime$. 
In particular, for an $m$-chart $\Gamma$ on $D$ presenting a 2-dimensional braid over $F$, we can deform $\Gamma$ to an $m$-chart on $D^\prime$ which presents a 2-dimensional braid over $F^\prime$, by a finite number of ambient isotopies of $\mathbb{R}^3$ and Roseman moves for surface diagrams with $m$-charts. 
\end{rem}

For a surface link $F$, the {\it braid index} of $F$, denoted by $\mathrm{Braid}(F)$, is the minimum number of degree of a 2-dimensional braid over the standard $S^2$ which is equivalent to $F$ as a surface link. 
We remark that this invariant exists for any surface link, since we assume that surface links are oriented and any oriented surface link can be presented in the form of a 2-dimensional braid over the standard $S^2$ \cite{Kamada94}.
The following proposition follows from Remark \ref{thm:6-3}.
\begin{proposition}\label{cor:6-4}
For a 2-dimensional braid $S$ over a surface link $F$ of degree $m$, we have
\[
\mathrm{Braid}(S) \leq m \cdot \mathrm{Braid}(F), 
\]
where we regard $S$ as a surface link. 
\end{proposition}

\section{Remarks}\label{sec7}
\subsection{C-moves}
In this paper, we do not treat C-moves, but we can add \lq\lq C-moves" after \lq\lq Roseman moves" in Corollary \ref{thm:6-2}. 
Two $m$-charts on a 2-disk $E$ are {\it C-move equivalent} if they are related by a finite sequence of ambient isotopies of $E$ rel $\partial E$ and CI, CII, CIII-moves \cite{Kamada92, Kamada02}; see \cite{Kamada02} for the precise definition and the illustrations of the moves. 
Two 2-dimensional braids over $E$ of degree $m$ are equivalent if and only if $m$-charts are C-move equivalent; see \cite{Kamada92, Kamada92-2,Kamada96, Kamada02}.

\subsection{Torus-covering links}\label{sec7-2}
A {\it torus-covering link} \cite{N} is a surface link in the form of a 2-dimensional braid over a standard torus $T$, hence  presented by an $m$-chart on $T$. In \cite{N2}, we showed an explicit method to deform a torus-covering link by ambient isotopies to the form of a 2-dimensional braid over the standard 2-sphere. In the proof, we described the equivalent deformation of knotted surfaces by using motion picture method. We gave a rather long explanation to describe the deformation and illustrate equivalent surfaces, since we regarded each surface $F$ as a one-parameter family of slices of $F$ by 3-spaces $\mathbb{R}^3 \times \{t\}$ ($t \in \mathbb{R}$). Now that we have Roseman moves for surface diagrams with $m$-charts, we can provide a much simpler proof of this result \cite[Theorem 3.2]{N2}. Further, we remark that Proposition \ref{cor:6-4} in this paper is a generalization of \cite[Corollary 4.1]{N2}.

\section{Well-definedness of Roseman moves}\label{sec8}
This section is devoted to showing Theorem \ref{prop:7-1}. We denote Roseman moves as illustrated in Fig. \ref{fig:0417-03} by $R$--$(i)$ ($i=1,\ldots,7$), and  
we denote the Roseman moves as illustrated in Fig. \ref{fig:0417-04} by $\Omega$--$(i)$ ($i=1,\ldots,6$), as indicated in the figures. 
 
\begin{proof}[Proof of Theorem \ref{prop:7-1}]

It is known \cite{Yashiro} that the Roseman move $R$--(2) is constructed from $R$--(1) and $R$--(4). Hence it suffices to show that the 2-dimensional braids presented by $R$--$(i)$ for $i \neq 2$ and $\Omega$--$(j)$ ($i=1,\ldots,7, j=1,\ldots,6$) are equivalent. 
The cases for $R$--(1), $\Omega$--(1), and $R$--(6) are shown by Lemmas \ref{lem1}, \ref{lem3}, and \ref{lem2}, respectively. 

The other cases are shown as follows. 
By an ambient isotopy of $\mathbb{R}^3$, we can assume that the surface diagrams are in the form of 2-dimensional braids over a 2-disk $E$. Then, the presented surfaces are also in the form of 2-dimensional braids over $E$. It is known (see \cite[Theorem 1.34]{CKS}, \cite{Kamada02}) that for $m$-charts on $E$ with exactly one black vertices, or, for such $m$-charts with no black vertex, the presented 2-dimensional braids over $E$ are equivalent if the boundaries are the same. Hence we can see the well-definedness for these cases.  
\end{proof}

\begin{lemma}\label{lem1}
Empty $m$-charts on the diagrams indicated by Roseman move $R$--$(1)$ present equivalent 2-dimensional braids.  
\end{lemma}
\begin{figure}[ht]
 \includegraphics*{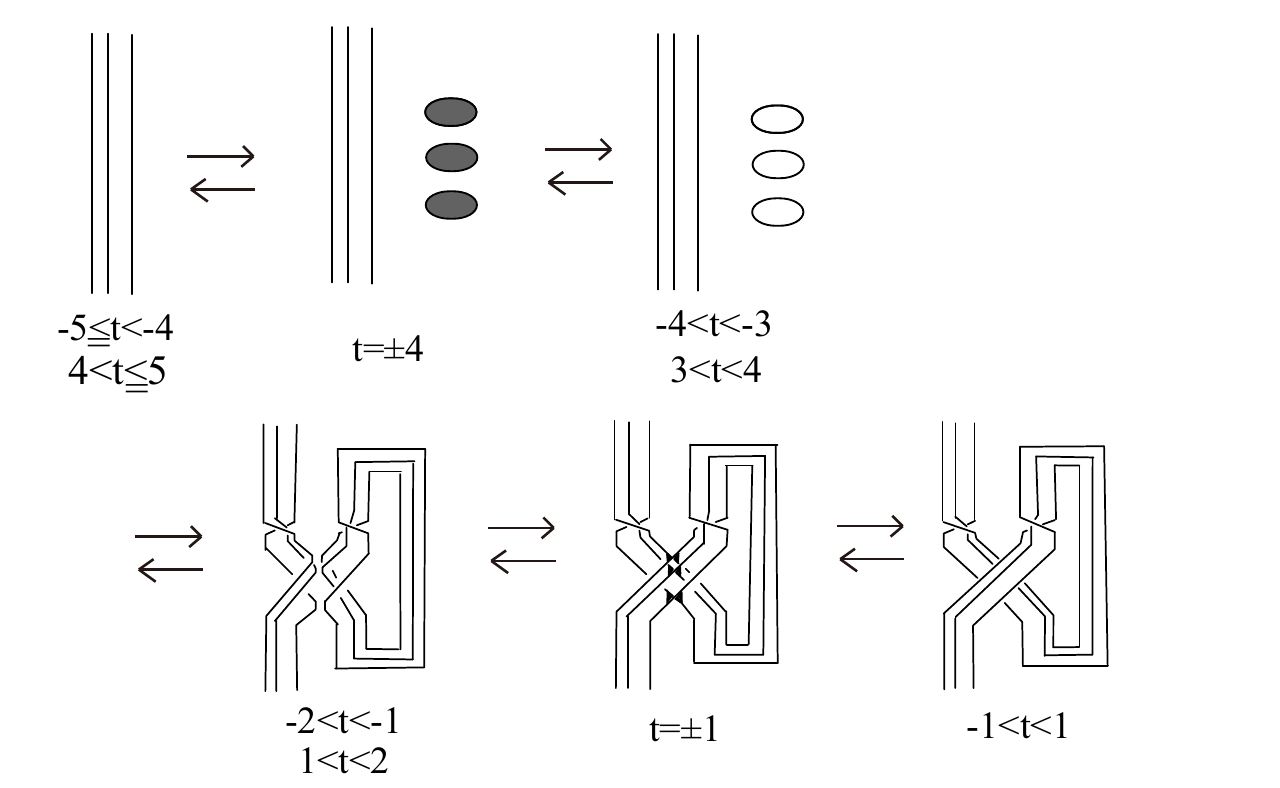}
\caption{The 2-dimensional braid presented by $R$--(1).}
\label{fig:0501-01}
 \end{figure}
\begin{figure}[ht]
 \includegraphics*{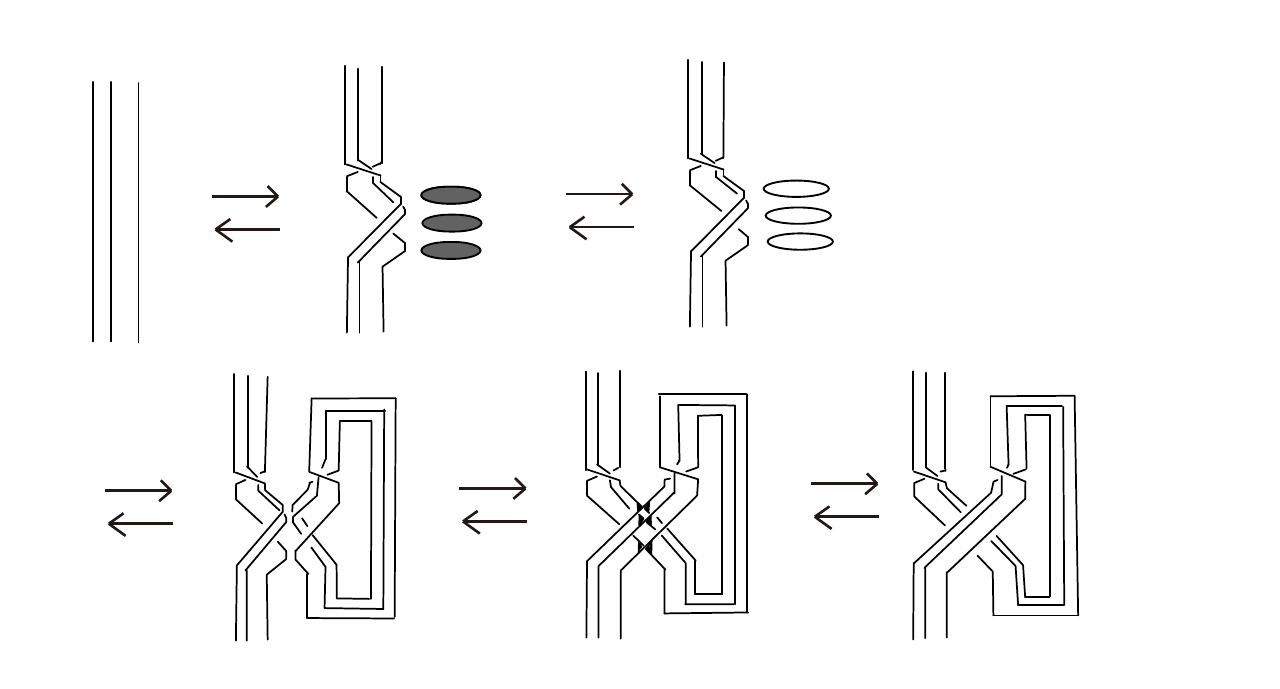}
\caption{The 2-dimensional braid presented by $R$--(1).}
\label{fig:0501-02}
 \end{figure}
\begin{figure}[ht]
 \includegraphics*{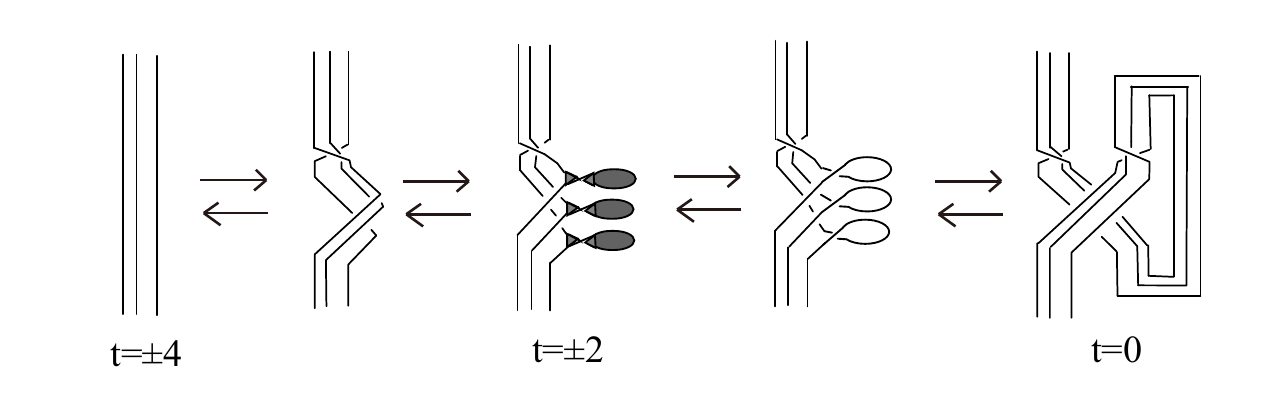}
\caption{The 2-dimensional braid presented by $R$--(1), 
where we assume that the disks at $t=\pm 2$ are branch points.}
\label{fig:0501-03}
 \end{figure}
\begin{proof}
We give a sketch of a proof by illustrating figures. Details are left to the reader; see \cite{N2} for a similar argument. 
By definition of the 2-dimensional braid presented by an empty $m$-chart, the 2-dimensional braid presented by an empty $m$-chart on the diagram indicated by the upper figure of $R$--$(1)$ in Fig. \ref{fig:0417-03} has the motion picture as in Fig. \ref{fig:0501-01}. (There seem to be two types, but one is transformed to the other by rotating the diagram around the vertical axis by $\pi$.) By ambient isotopies, we can deform the surface to the form presented by the motion picture as in Fig. \ref{fig:0501-02}, and then to the form as in Fig. \ref{fig:0501-03}. 
For the resulting surface $S$, we can assume that the disks at $t=\pm 2$ are branch points. We denote each picture of the motion picture of $S$ by $S_t$, and we assume that $S_t$ is in a 3-space $\mathbb{R}^3 \times \{t\}$.  
Figure \ref{fig:0501-03} indicates the existence of an ambient isotopy $\{g_u\}_{u \in [0,1]}$ of $\mathbb{R}^3$ which carries $S_{\pm 4}$ to $S_0$ such that $S\cap (\mathbb{R}^3 \times [-4,0])$ (resp. $S\cap (\mathbb{R}^3 \times [0,4])$) is presented as the orbit of $g_u(S_{-4})$ (resp. $g_{1-u}(S_{4})$), where $u \in [0,1]$. 
Hence we can construct an ambient isotopy of $\mathbb{R}^4$ which carries $S$ to the trivial 2-dimensional braid, which is presented by an empty $m$-chart on the diagram indicated by the lower figure of $R$--(1). 
Thus $R$--(1) is well-defined.
\end{proof}

\begin{lemma}\label{lem3}
 The $m$-charts on the diagrams indicated by Roseman move $\Omega$--$(1)$ present equivalent 2-dimensional braids.  
\end{lemma}

\begin{proof}
We denote by $D$ the surface diagram indicated by $\Omega$--(1), and we denote by $F$ the surface of the companion in $\mathbb{R}^4$ presented by $D$. 
Let us denote by $N(l)$ (resp. $N(l^\prime)$) the preimage by the projection $N(F) \to F$ of the path $l$ (resp. $l^\prime$) in $F$ indicated by the edge of the $m$-chart in the upper figure (resp. lower figure) of $\Omega$--$(1)$ in Fig. \ref{fig:0417-04}. 
The ambient isotopy $\{g_u\}$ in the proof of Lemma \ref{lem1} indicates the existence of an ambient isotopy of $\mathbb{R}^4$ which fixes the framing of $N(F)$ determined by an empty $m$-chart on $D$ (see Remark \ref{rem:5-3}) and carries $N(l)$ to $N(l^\prime)$.  
By this ambient isotopy, we can relate the presented 2-dimensional braids, and hence, they are equivalent.  
\end{proof} 

\begin{lemma}\label{lem2}
 Empty $m$-charts on the diagrams indicated by Roseman move $R$--$(6)$ present equivalent 2-dimensional braids.  
\end{lemma}
\begin{figure}[ht]
\begin{center}
 \includegraphics*{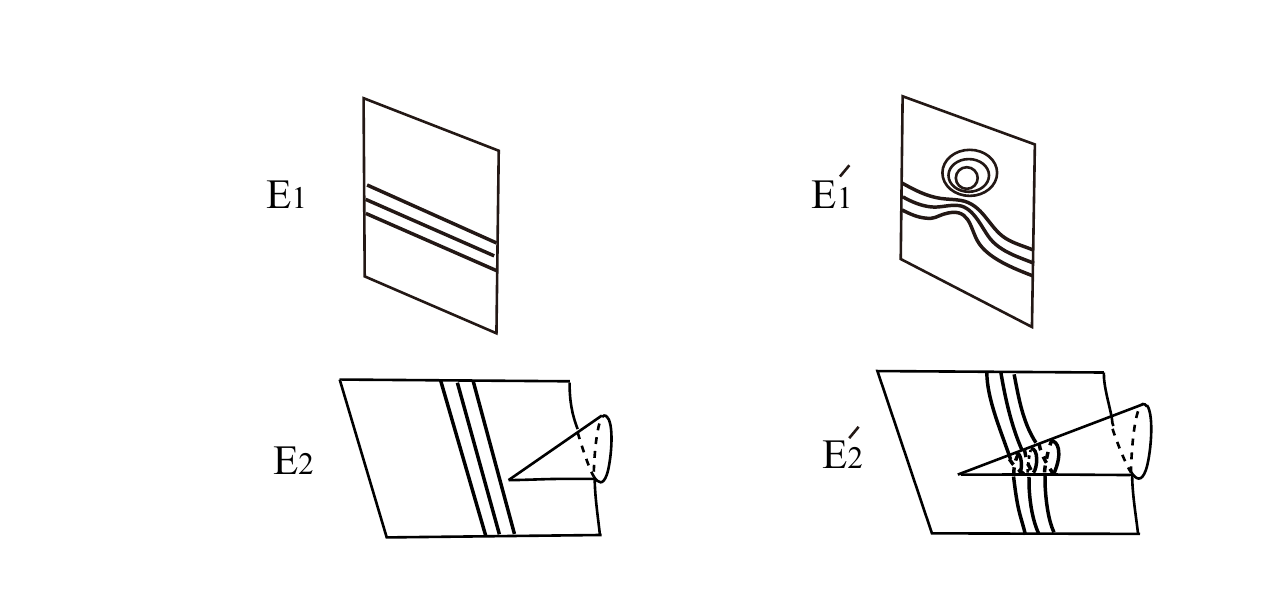}
\end{center}
\caption{Surface diagrams with $m$-charts. For simplicity, we omit the over/under information of sheets, and orientations and  labels of edges of $m$-charts. The parallel edges of $m$-charts present the product of an interval and an $m$-braid with a negative or positive half twist.}
\label{2013-1211-01}
 \end{figure}

\begin{proof}
 Let us denote by $E_1$ (resp. $E_2$) the surface diagram with an $m$-chart presenting a connected surface without the branch point (resp. with the branch point) in the upper figure of $R$--(6) in Fig. \ref{fig:0417-03}, and let us denote by $E_1^\prime$ (resp. $E_2^\prime$) the surface diagram with an $m$-chart in the lower figure which is deformed from $E_1$ (resp. $E_2$). 
Let us pick up $E_i$ (resp. $E_i^\prime$) $(i=1,2)$ separately; note that there appear edges of an $m$-chart around the curve $l$ (resp. $l^\prime$) which was the intersection of $E_1$ and $E_2$ (resp. $E_1^\prime$ and $E_2^\prime$). On $E_1$ and $E_2$, around $l$ which was a double point curve, there appear parallel edges of $m$-charts which present the product of an interval and an $m$-braid with a negative or positive half twist according to the over/under information of $E_1$ and $E_2$, and hence the $m$-charts are as in the left figure of Fig. \ref{2013-1211-01}. On $E_1^\prime$ and $E_2^\prime$, around a part of $l^\prime$ which was a double point curve, there appear parallel edges of $m$-charts similar to the case of $l$, and around a part of $l^\prime$ which was a triple point, the edges of the $m$-charts contain no black vertices.  
Since an $m$-braid with a half twist is preserved by the homomorphism which sends $\sigma_i$ to $\sigma_{m-i}$ ($i=1,\ldots,m-1$), 
we can assume that the $m$-charts are as in the right figure of Fig. \ref{2013-1211-01}. 

We denote by $S_i$ and $S_i^\prime$ $(i=1,2)$ the 2-dimensional braids presented by $E_i$ and $E_i^\prime$ with the $m$-charts, respectively. 
Since the 2-dimensional braid presented by the upper figure (resp. lower figure) of $R$--(6) are the split union of $S_1$ and $S_2$ ($S_1^\prime$ and $S_2^\prime$), it suffices to show that 
$S_i$ and $S_i^\prime$ are equivalent for $i=1,2$. 
Since $S_1$ and $S_1^\prime$ are 2-dimensional braids over a 2-disk with no branch points and the boundaries are the same, by the same reason as in the proof of Theorem \ref{prop:7-1}, $S_1$ and $S_1^\prime$ are equivalent. 
Since the boundaries of the $m$-charts on $E_2$ and $E_2^\prime$ are the same, one diagram with the $m$-chart is deformed to the other by applying $\Omega$--(1) several times, and it follows from Lemma \ref{lem3} that $S_2$ and $S_2^\prime$ are equivalent. Thus we have the required result.
\end{proof}

\section*{Acknowledgments}
The author would like to thank Professor Seiichi Kamada and the referee for their helpful comments.
The author was supported by JSPS Research Fellowships for Young Scientists ($24 \cdot 9014$), and iBMath through the fund for Platform for Dynamic Approaches to Living System from MEXT.

\end{document}